\documentclass[12pt,a4paper,oneside]{amsart}
\usepackage{a4,a4wide}
\usepackage{amsfonts,amsmath,amssymb,amsthm,enumitem}
\usepackage{mathtools}
\usepackage{esint}

\usepackage{color}
\usepackage{ifpdf}
\ifpdf
    \usepackage[pdftex]{graphicx}
    \DeclareGraphicsExtensions{.png,.pdf,.jpg}
\else
   \usepackage[dvips]{graphicx}
   \DeclareGraphicsExtensions{.eps}
\fi
\graphicspath{{.}{figures/}}
\usepackage[latin1]{inputenc}
\numberwithin{equation}{section}

\usepackage{doi}

\usepackage{tikz}
\usetikzlibrary{quotes,angles,positioning}

\newtheorem{theorem}{Theorem}[section]
\newtheorem{lemma}[theorem]{Lemma}

\newtheorem{remark}[theorem]{Remark}
\newtheorem{question}[theorem]{Question}
\newtheorem*{theorem*}{Theorem}
\newtheorem*{lemma*}{Lemma}
\newtheorem*{proposition*}{Proposition}
\newtheorem*{corollary*}{Corollary}
\renewcommand\tilde{\widetilde}

\def\R{\mathbb{R}}

\def\EE{\mathbb{E}}

\renewcommand{\phi}{\varphi}

\def\1{\mathbf{1}}

\def\XXint#1#2#3{{\setbox0=\hbox{$#1{#2#3}{\int}$ }
\vcenter{\hbox{$#2#3$ }}\kern-.57\wd0}}

\def\eps{\varepsilon}

\newcommand{\bra}[1]{\left( #1 \right)}

\newcommand{\cur}[1]{\left\{ #1 \right\}}

 \newcommand{\abs}[1]{\left\vert#1\right\vert}

\newcommand{\mres}{\mathbin{\vrule height 1.6ex depth 0pt width
0.13ex\vrule height 0.13ex depth 0pt width 1.3ex}}

\begin{document}
\title[There is no stationary $p$-c. m. Poisson matching in 2D]{There is no stationary $p$-cyclically monotone Poisson matching in 2D}

%\author{Francesco Mattesini}
\author{  Martin Huesmann \address[Martin Huesmann]{Universit\"at M\"unster,  Germany} \email{martin.huesmann@uni-muenster.de} \hspace*{0.5cm} Francesco Mattesini \address[Francesco Mattesini]{Universit\"at M\"unster \& MPI Leipzig,  Germany} \email{francesco.mattesini@uni-muenster.de}  \hspace*{0.5cm} Felix Otto \address[Felix Otto]{MPI Leipzig, Germany} \email{Felix.Otto@mis.mpg.de} } 
\thanks{All authors are supported by the Deutsche Forschungsgemeinschaft (DFG, German Research Foundation) through the SPP 2265 {\it Random Geometric Systems}. MH and FM have been funded by the Deutsche Forschungsgemeinschaft (DFG, German Research Foundation) under Germany's Excellence Strategy EXC 2044 -390685587, Mathematics M\"unster: Dynamics--Geometry--Structure. FM has been funded by the Max Planck Institute for Mathematics in the Sciences.}

\begin{abstract}
%We extend the previous result from \cite{HMO21} showing that for $p>1$ there is no $p$-cyclically monotone stationary matching of two independent Poisson processes in dimension $d=2$. The proof combines the $p$-harmonic approximation result from \cite[Theorem 1.1]{koch23} with local asymptotics for the two-dimensional matching problem. Moreover, we prove the local upper bounds in the case $p>1$ which, to the best of our knowledge, are not readily available in the current literature.
We show that for $p>1$ there is no $p$-cyclically monotone stationary matching of two independent Poisson processes in dimension $d=2$. The proof combines the $p$-harmonic approximation result from \cite[Theorem 1.1]{koch23} with local asymptotics for the two-dimensional matching problem. 
Moreover, we prove a.s.\ local upper bounds of the correct order in the case $p>1$, which, to the best of our knowledge, are not readily available in the current literature.
\end{abstract}

\date{\today}
\maketitle

\section{Introduction}
%\cm{TO DO:\\
%$\rightsquigarrow$ connection to the work of Gutierrez-Montanari for the $L^\infty$ estimate \cite{Gut22}\\
%$\rightsquigarrow$ notation for the distance on the torus}
%In the following I will denote by $W_{p, \Omega}$ the $p$-Wasserstein distance restricted to $\Omega$.
Let $\{X\}, \{Y\} \subset \mathbb{R}^d$ be two locally finite\footnote{and thus countable} random point sets. We consider their matching, that is a (random) bijection from $\{X\}$ to $\{Y\}$. More precisely we will focus on dimension $d=2$ and we are primarily interested in the case where the two random point sets are given by two independent Poisson point process of unit intensity and the map $T$ is $p$-locally optimal for $p \ge 1$, meaning that for any other bijection $\tilde T$ that differs from $T$ only on a finite number of points
\begin{equation}\label{eq:plocopt}
\sum_X (|T(X) - X|^p - |\tilde{T}(X) - X|^p) \le 0. 
\end{equation}
Since the sum in \eqref{eq:plocopt} is finite, the latter provides a natural connection to the optimal transport problem between the measures 
\begin{equation}\label{eq:empmeasp}
\mu := \sum_X \delta_X \quad \text{and}\quad \nu := \sum_Y \delta_Y
\end{equation}
related via $T_{\#} \mu = \nu$. However, note that the map $T$ cannot be viewed as a usual minimizer in the optimal transport problem due to the (typically) infinite number of points.

\medskip
Let us now be more specific on the random setting we consider. We assume that the $\sigma$-algebra generated by $(\{X\},\{Y\},T)$ is rich enough so that the numbers of matched pairs $(X,Y) \in U\times V$ of any two Lebesgue-measurable sets $U,V \subset \mathbb{R}^d$ (with $U$ or $V$ having finite Lebesgue measure\footnote{so that the following number is finite})
\[
N_{U,V} := \# \{ (X,Y) \in U \times V \;|\;Y=T(X)\} \in \{0,1,\dots\}
\]
are measurable. Moreover, we assume that the law of the triple $(\{X\},\{Y\},T)$ is stationary, that means it is invariant under the action of the additive group $\mathbb{Z}^d$
\begin{equation}\label{eq:actionadgroup}
(\{X\},\{Y\},T) \mapsto (\{\bar x + X\},\{\bar x + Y\},T(\cdot - \bar x)+\bar x) \quad \text{for $\bar x \in \mathbb{Z}^d$.}
\end{equation}
Note that stationarity is a structural assumption which will allow us to say that for any shift vector $\bar x$, the random natural numbers $N_{\bar x + U, \bar x+ V}$ and $N_{U,V}$ have the same distribution. Furthermore, we make the assumption that the action \eqref{eq:actionadgroup} is ergodic. 
%As a consequence of the last assumption, an application of the Birkhoff's ergodic theorem implies
%\[
%\lim_{R \uparrow \infty} \frac1{{R}^d} \sum_{\bar x \in \mathbb{Z}^d \cap [0,R)^d} N_{\bar x + U, \bar x+ V} = \mathbb{E} N_{U,V} \quad \text{almost surely}.
%\]

\medskip
The aim of this paper is to explore the geometric properties of the matching $T$ between two independent Poisson point processes in  dimension $2$. A matching in $\R^2$ is called {\it planar} if for any choice of points $X,X'$, the line segments connecting $X$ to $T(X)$ and $X'$ to $T(X')$ do not intersect. In 2002 the following question was proposed by Peres in \cite{HoPePeSc09}. 
\begin{question}\label{question1}
For two independent Poisson processes of intensity one does there exist a stationary planar matching?
\end{question}
It has been shown by Holroyd in \cite{Ho11}, that there is no translation-invariant planar matching on the strip $\R\times [0,1)$. Yet,  Question \ref{question1} is still unsolved in $\R^2$ and it is far from clear what its answer should be. Again in \cite{Ho11}, it was observed by Holroyd, that by the triangle inequality the $1$-local optimality condition \eqref{eq:plocopt} implies planarity. Hence, it is natural to consider the following modification of Question \ref{question1}, which has been proposed in \cite{HoJaWa20}.
\begin{question}\label{question2}
For two independent Poisson processes of intensity one does there exist a stationary and $p$-locally optimal matching?
\end{question}
Question \ref{question2} is well understood for dimension $d=1$ and $d \ge 3$, e.g.\ see \cite[Theorem 2 and Theorem 7]{HoJaWa20} or \cite{Hu16}.  Nevertheless,  the two dimensional setting is partially unsolved. In \cite[Theorem 7]{HoJaWa20} it has been shown that stationary $p$-locally optimal matchings exist for $p < 1$. Our result, together with the one obtained in \cite{HMO21}, complements the one of \cite{HoJaWa20} in the regime $p>1$ and $d=2$, leaving unsolved the case $p=1$ and of course the question on planarity.
%
%extend the non existence result of \cite{HMO21} to the setting of $p$-minimal matching.  To this end we recall the terminology and basic assumptions on the random ensemble we consider. 

\begin{theorem}\label{thm:main}
For $d = 2$ and $p>1$, there exists no stationary and ergodic ensemble of $(\{X\},\{Y\},T)$, where $\{X\}, \{Y\}$ are independent Poisson point processes and $T$ is a $p$-cyclically monotone bijection of $\{X\}$ and $\{Y\}$.
\end{theorem}
Before commenting on the proof of Theorem \ref{thm:main} let us add  some remarks on extensions and variants of Theorem \ref{thm:main}. 
\begin{remark}
Theorem \ref{thm:main} remains true if we replace the bijection $T$ by the a priori more general object of a stationary coupling $Q$.
This can be seen for instance by directly writing the proof in terms of couplings which essentially only requires notational changes.
\end{remark}

\begin{remark}\label{rem:remark1.3}
Natural variants of stationary matchings are given by stationary allocations of a point process $\{X\}$, i.e.\ a stationary map $T:\R^d\to\{X\}$ such that $\mathsf
{Leb}(T^{-1}(X))$ equals $\EE[\#\{X\in(0,1)^d\}]^{-1}$, e.g.\ see \cite{HoHoPe06, ChPePeRo10,MaTi16, HuSt13}.

Mimickicking the proof of Theorem \ref{thm:main}, one can show that in $d=2$ there is no locally $p$-optimal stationary allocation to a Poisson process. The only place which will require minor changes is the $L^\infty$ estimate Lemma \ref{lem:Linfty}.
\end{remark}

\begin{remark}
Since by ergodicity and stationarity we can argue on a pathwise level via the $p$-harmonic approximation Theorem (cf.\ Section \ref{sec:mainsteps}), we do not use many particular features of the Poisson point processes $\mu$ and $\nu$ in the proof of Theorem \ref{thm:main}.

We mainly use two properties: The first property is concentration around the mean. The second property is more involved. Denote by $W_p$ the $L^p$ Wasserstein distance. We use that\footnote{Given a measure $\mu$ on a $\mathbb{R}^d$ and a subset $\Gamma \subseteq \mathbb{R}^d$ we denote its restriction to $\Gamma$ by $\mu \mres \Gamma(\ \cdot \ ) := \mu(\ \cdot \ \cap \Gamma)$.}
$\frac1{R^d}W_p(\mu \mres B_R,\frac{\mu(B_R)}{|B_R|}\mathsf{Leb})$ diverges at the same rate for $R\to\infty$ as $\frac{1}{R^d} W_{p-\eps}(\mu \mres B_R,\frac{\mu(B_R)}{|B_R|}\mathsf{Leb})$
for some $\eps>0$.
\end{remark}

The proof of Theorem \ref{thm:main} goes along the same lines as in \cite[Theorem 1.1]{HMO21}, see also \cite[Section 1.1]{HMO21}. We already remark here that there are two new ingredients: The $p$-harmonic approximation theorem and almost sure upper asymptotics for the matching cost. The former, already shown in \cite[Theorem 1.1]{koch23}, states that the displacement $T(X) - X$ is close (in the $p$-norm distance) to a $p'$-harmonic\footnote{we denote by $p'$ the conjugate exponent of $p$, i. e. $\frac1p + \frac1{p'} = 1$} gradient field $|\nabla \Phi|^{p'-2}\nabla \Phi$ provided that we are in a perturbative regime, which is quantified in terms of smallnes of the local energy
\begin{equation}\label{eq:energytermp}
E_p (R) := \frac 1 {R^d} \sum_{X \in B_R \;\text{or}\;T(X) \in B_R} |T(X) - X|^p
\end{equation}
and of the data term, that is the distance of $\mu \mres B_R$ and $\nu \mres B_R$ to the Lebesgue measure on the ball $B_R$ of radius $R$\footnote{we tacitly identify the (random) number density $n_\mu$ with the uniform measure $n_\mu \mathrm{d} x$.}
\begin{equation}\label{eq:datatermp}
D_p(R) := \frac1{R^d} W^p_{p,B_R}(\mu, n_\mu) + \frac{R^p}{n_\mu}(n_\mu - 1)^p + \frac1{R^d} W^p_{p,B_R}(\nu, n_\nu) + \frac{R^p}{n_\nu}(n_\nu - 1)^p,
\end{equation}
where $n_\mu = \frac{\# \{X \in B_R\}}{|B_R|}$ and $n_\nu = \frac{\# \{Y \in B_R\}}{|B_R|}$, and $W_{p,\Gamma} (\mu, \nu) = W_p (\mu \mres \Gamma, \nu \mres \Gamma)$ for a Borel set $\Gamma \subset \R^d$. The latter, which we state here, is our second main result and concerns concentration properties of the matching cost. 

\begin{theorem}\label{lem:uppdata}
Let $\mu, \nu$ denote two independent Poisson point processes in $\mathbb{R}^d$ of unit intensity. There exists a constant $C$, and a random radius $r_*<\infty$ a.~s. such that for a (random) sequence of approximately dyadic radii\footnote{we say that a radius $R$ is approximately dyadic if there exists a dyadic radius $R'$ and a constant $C\in\big(\frac12,2\big)$ such that $R = C R'$.} $R \ge r_*$
\begin{equation}\label{eq:thmuppdata}
D_p (R) \le C \begin{cases}
\ln^\frac p2 R & \text{if $d=2$},\\
1 & \text{if $d\ge 3$}.
\end{cases}
\end{equation}
\end{theorem}
We remark here that by the annealed (i. e. in expectation, see for instance \cite{AKT84}, \cite{BaBo13}) results for the matching problem in dimension $d=2$ and by the concentration properties of the Poisson process we may expect that $D_p(R) \le O(\ln^\frac p2 R)$. However, the standard arguments based on concentration of measures to improve the annealed result already available in the literature to an almost sure one fail whenever $p > d$, see also  \cite[Remark 6.5]{GolTre} for a discussion of the problem in the setting of strong convergence of asymptotic costs. In order to prove Theorem \ref{lem:uppdata} we make use of the dynamical formulation of optimal transport, which allows us to combine PDE arguments together with the already existing concentration arguments for the Poisson point process, see Section \ref{sec:uppbounddata}.

\subsection{Main steps in the proof of Theorem \ref{thm:main}}
%\cm{FM: choose square or balls}
We describe here the main steps in the proof of Theorem \ref{thm:main}. 
For the detailed proofs we refer the reader to Section \ref{sec:proofscoup}.

\medskip
Following the arguments of the proof of \cite[Theorem 1.1]{HMO21}, we argue by contradiction. We show that for a locally $p$-optimal stationary matching $T$ between $\{X\}$ and $\{Y\}$ we have the upper bound 
\begin{equation}\label{eq:mainstepupp}
\frac 1{R^d} \sum_{X \in B_R \;\text{or}\;T(X) \in B_R} |T(X) - X| \le o(\ln^\frac12 R),
\end{equation}
and the lower bound (see \cite[Lemma 2.4]{HMO21})
\[
\frac 1{R^d} \sum_{X \in B_R \;\text{or}\;T(X) \in B_R} |T(X) - X| \ge \Omega(\ln^\frac12 R),
\]
implying the desired contradiction. 

\medskip
We now describe the main steps and the main differences between the proof of the upper bound \eqref{eq:mainstepupp} in the general case $p>1$ and in the quadratic case. The first common step is the observation that by stationarity and ergodicity the number of Poisson points which are transported by a far distance is small in volume fraction, i.~e. the following $L^0$-estimate on the displacement holds
\begin{equation}\label{eq:mainstepL0}
\# \{ X \in (-R,R)^d \;:\; |T(X)-X| \gg 1\} \le o(R^d),
\end{equation}
see \cite[Lemma 2.1]{HMO21} for a precise statement. \emph{As in the quadratic case this will be the only place where stationarity and ergodicity enter}. The next step consists on improving the ergodic estimate \eqref{eq:mainstepL0} to a uniform bound. As opposed to \cite[Lemma 2.2]{HMO21} we cannot rely on the monotonicity of the map $T$. However, the local $p$-optimality of the matching $T$ allows us to exploit the geometry of its support to improve \eqref{eq:mainstepL0} to
\begin{equation}\label{eq:mainstepLinfty}
 |T(X)-X| \le o(R) \text{\;provided that $X \in (-R,R)^d$,}
 \end{equation} 
see Lemma \ref{lem:Linfty}. By concentration properties of the Poisson process we may assume that $\frac{\#\{X \in B_R\}}{|B_R|} \in [\frac12,2]$ for $R \gg 1$. Summing \eqref{eq:mainstepLinfty} over $B_R$ we obtain
\begin{equation}\label{eq:mainstepenergy}
\frac1{R^d} \sum_{X \in B_R}|T(X) - X|^p \le o(R^p).
\end{equation}
The bound \eqref{eq:mainstepenergy} will let us run the harmonic approximation argument already employed in \cite[Lemma 2.3]{HMO21}. Nevertheless, in the current setting we need a \emph{$p$-cost version} of the latter. By the $p$-harmonic approximation theorem \cite[Theorem 1.1]{koch23} we know that if the energy term \eqref{eq:energytermp} and the data term \eqref{eq:datatermp} are small then  the displacement $T(X) - X$ is close to a $p'$-harmonic gradient field. %Nevertheless, in the current setting we need a \emph{$p$-cost version} of the latter. \cm{QUI QUI QUI}%In \cite[Theorem 1.1]{koch23} it has been shown that the displacement $T(X) - X$ is close (in the $p$-norm distance) to a $p'$-harmonic\footnote{we denote by $p'$ the conjugate exponent of $p$, i. e. $\frac1p + \frac1{p'} = 1$} gradient field $|\nabla \Phi|^{p'-2}\nabla \Phi$ provided that we are in a perturbative regime, which is quantified in terms of smallnes of the local energy
%\begin{equation}\label{eq:energytermp}
%E_p (R) := \frac 1 {R^d} \sum_{X \in B_R \;\text{or}\;T(X) \in B_R} |T(X) - X|^p
%\end{equation}
%and of the data term, that is the distance of $\mu \mres (-R,R)^d$ and $\nu \mres (-R,R)^d$ to the Lebesgue measure on the square $(-R,R)^d$
%\begin{equation}\label{eq:datatermp}
%D_p(R) := \frac1{R^d} W^p_{p,(-R,R)^d}(\mu, n_\mu) + \frac{R^p}{n_\mu}(n_\mu - 1)^p + \frac1{R^d} W^p_{p,(-R,R)^d}(\mu, n_\mu) + \frac{R^p}{n_\nu}(n_\nu - 1)^p,
%\end{equation}
%where $n_\mu = \frac{\# \{X \in (-R,R)^d\}}{(2R)^d}$ and $n_\nu = \frac{\# \{Y \in (-R,R)^d\}}{(2R)^d}$, and $W_{p,B} (\mu, \nu) = W_p (\mu \mres B, \nu \mres B)$. 
By \eqref{eq:mainstepenergy} and by exchanging the roles of $\{X\}$ and $\{Y\}$ in \eqref{eq:mainstepenergy} we have $E_p(R) \le o(R^p)$. 
On the other hand Theorem \ref{lem:uppdata} ensures that $D_p(R) \le O(\ln^\frac p2 R)$.
%By the annealed (i. e. in expectation) results for the matching problem in dimension $d=2$ and by the concentration properties of the Poisson process we may expect that $D_p(R) \le O(\ln^\frac p2 R)$. However, the standard arguments based on concentration of measure to improve the annealed result already available in the literature to an almost sure one fail whenever $p > d$, see also  \cite[Remark 6.5]{GolTre} for a discussion of the problem in the setting of strong convergence of asymptotic costs. However, we show the desired asymptotic bound for $D_p$ via a PDE approach, see Lemma \ref{lem:uppdata}. 
Hence, we are in a position to iteratively exploit the $p$-harmonic approximation result on an increasing sequence of scales to obtain that the local energy inherits the asymptotic of the data term $D_p$:
\begin{equation}\label{eq:mainstepharmapp}
\frac{1}{R^d} \sum_{X \in B_R \;\text{or}\;T(X) \in B_R} |T(X) - X|^p \le O(\ln^\frac p2 R),
\end{equation}
see Lemma \ref{lem:harmonicapprox}. Combining this with the $L^0$-estimate as in \cite[Lemma 2.4]{HMO21} yields
\[
\frac1{R^d} \sum_{X \in B_R \;\text{or}\;T(X) \in B_R} |T(X) - X| \le o(\ln^\frac12 R),
\]
see Lemma \ref{lem:upperboundL1}.

%%%%%%%%%
%%%%%%%%%
%%%%%%%%%
%%%%%%%%%
\section{Proofs}\label{sec:proofscoup}
\subsection{Upper Bound}\label{sec:uppbounddata}
In this section we establish the upper bound asymptotics for the data term \eqref{eq:datatermp}. 
%\begin{lemma}\label{lem:uppdata}
%Let $\mu, \nu$ denote two independent Poisson point processes in $\mathbb{R}^2$ of unit intensity. There exists a constant $C$, and a random radius $r_*<\infty$ a.~s. such that for any dyadic radii $R \ge r_*$
%\[
%D_p (R) \le C \ln^\frac p2 R.
%\]
%\end{lemma}
The proof of Theorem \ref{lem:uppdata} will follow from the upper bound asymptotics of the distance between the Poisson point process on a torus $[0,R)^d$ and the Lebesgue measure. Given two measures $\mu, \nu$ on $\R^d$, we consider their projection on the torus $[0,R)^d$ and we denote by $\tilde{W}_{[0,R)^d; p}(\mu, \nu)$ the $p$-Wasserstein distance on the torus $[0,R)^d$ between them. Given a Borel set $\Gamma \subset [0,R)^d$, we denote by $\tilde{W}_{[0,R)^d;p,\Gamma} (\mu, \nu) = \tilde{W}_{[0,R)^d;p} (\mu \mres \Gamma, \nu \mres \Gamma)$ its restriction to $\Gamma$.
\begin{lemma}\label{lem:uppbou}
Let $\mu$ be a Poisson point process on the torus $[0,R)^d$ of unit intensity. There exists a constant $C$, a random radius $r_*<\infty$ a.~s. such that for any dyadic radii $R \ge r_*$ and any $p \ge 1$
\begin{equation}\label{eq:uppbound}
\tilde{W}_{[0,R)^d; p}^p (\mu, n) \le CR^d \begin{cases}  \ln^\frac p2 R & \text{if $d=2$},\\
1 & \text{if $d\ge3$},
\end{cases}
\end{equation}
where $n = \frac{\mu([0,R)^d)}{R^d}$ is the (random) number density.
\end{lemma}
We shall derive Theorem \ref{lem:uppdata} combining Lemma \ref{lem:uppbou} with a restriction result for the data term, which will allow us to exchange the periodic Wasserstein distance with the Euclidean one. 
\begin{lemma}\label{lem:localization}
For any positive measure $\mu$ on the torus $[-2R,2R)^d$ there exists a constant $C>0$ such that
\[
\int_{R-\frac12}^{R+\frac12} \bigg( \tilde{W}_{[-4 R,4 R)^d; p,B_{\bar R}}^p (\mu, n_{\bar R}) + \frac{(n_{\bar R}-1)^p}{n_{\bar R}} \bigg) \, d\bar R \le C \tilde D,
\]
provided 
\[
\tilde{D}:= \tilde{W}_{[-4R,4R)^d; p}^p(\mu, n) + \frac{(n-1)^p}n\ll 1,
\]
where $n_{\bar R} = \frac{\mu(B_{\bar R})}{|B_{\bar R}|}$ and $n = \frac{\mu([-4R,4R)^d)}{(8R)^d}$ are the (random) number densities.
\end{lemma}
The latter can be derived combining the proof of \cite[Lemma 2.10]{GHO} and \cite[Lemma 6.1]{koch23}.

\begin{proof}[Proof of Theorem \ref{lem:uppdata}]
%W.l.o.g. by exchanging the role of $\mu$ and $\nu$ it suffices to show that for the Poisson point process $\mu$ there exists a random radius $r_*$ such that for a (random) sequence of approximately dyadic radii $R \ge r_*$\footnote{We use the notation $A \lesssim B$ if there exists a global constant $C>0$, which may only depend on $d$, such that $A\le CB$. We write $A\sim B$ if both $A \lesssim B$ and $B \lesssim A$ hold.}
%\begin{equation}\label{eq:step0uppbou}
%\frac1{R^d} W_{p, B_R}^p (\mu, n_\mu) + \frac{R^p}{n_\mu} (n_\mu - 1)^p \lesssim \ln^\frac p2 R.
%\end{equation}
%The statement will follow by choosing the maximum of this random radius and the one pertaining to $\nu$.

%{\color{red} Why can we choose one sequence which satisfies (2.2) for both $\mu$ and $\nu$?}

\medskip
{\sc Step 1}.  Let $R \ge 1$ be an increasing sequence of approximately dyadic radii. We claim that there exists a constant $C$ and a random radius $r_* < \infty$ a.~s. such that for the fixed sequence of dyadic radii $R \ge r_*$
\begin{equation}\label{eq:step2uppbou}
\frac{R^p}{n_\mu}|n_\mu-1|^p \le C \begin{cases}
\ln^\frac p2 R & \text{if $d=2$},\\
1 & \text{if $d \ge 3$},
\end{cases}
\end{equation}
and
\begin{equation}\label{eq:step2uppbounu}
\frac{R^p}{n_\nu}|n_\nu-1|^p \le C \begin{cases}
\ln^\frac p2 R & \text{if $d=2$},\\
1 & \text{if $d \ge 3$}.
\end{cases}
\end{equation}
W.l.o.g. we focus on \eqref{eq:step2uppbou}. Indeed, \eqref{eq:step2uppbounu} will follow from \eqref{eq:step1uppbou} exchanging the role of $\mu$ and $\nu$ and taking the maximum of this radius and the one pertaining $\nu$. Since for large $R \gg 1$ we have $\frac{\ln^\frac p2 R}{R^p} \ll 1$ \eqref{eq:step2uppbou} is equivalent to \footnote{We use the notation $A \lesssim B$ if there exists a global constant $C>0$, which may only depend on $d$, such that $A\le CB$. We write $A\sim B$ if both $A \lesssim B$ and $B \lesssim A$ hold.}
\[
R^p |n_\mu-1|^p \lesssim \begin{cases}
\ln^\frac p2 R & \text{if $d=2$},\\
1 & \text{if $d \ge 3$}.
\end{cases}
\]
Since $n_\mu|B_R|$ is Poisson distributed with parameter $|B_R|$ by Cram\'er-Chernoff's bounds \cite[Theorem 1]{BouLuMa} we get for $d=2$
\[
\mathbb{P}(R^p|n_\mu-1|^p > \ln^\frac p2 R) = \mathbb{P}(|n_\mu|B_R| - |B_R||> C R \ln^\frac12 R) \lesssim \exp(-C \ln R)
\]
and
\[
\mathbb{P}(R^p|n_\mu-1|^p > C) = \mathbb{P}(|n_\mu|B_R| - |B_R||> C R^{d-1} ) \lesssim \exp(-C R^{d-2}),
\]
for $d\ge 3$. Finally, by a Borel-Cantelli argument \eqref{eq:step2uppbou} holds for the fixed sequence of approximately dyadic radii $R \ge r_*$. %{\color{red} MH: We have to be careful with the wording. It will not hold for all these sequences at the same time. We can choose a (or countably many) such sequences and then it holds a.s. for these sequences.}

\medskip
{\sc Step 2}. We claim that there exist a constant $C$ and a random radius $r_* < \infty$ a.~s. such that for a (random) sequence of approximately dyadic radii $R \ge r_*$
\begin{equation}\label{eq:step1uppbou}
\frac1{R^d} W_{p, B_R}^p (\mu, n_\mu) + \frac1{R^d} W_{p, B_R}^p (\nu, n_\nu) \le C \begin{cases}  \ln^\frac p2 R & \text{if $d=2$},\\
1 & \text{if $d\ge3$}.
\end{cases}
\end{equation}
By Lemma \ref{lem:uppbou} we may assume that \eqref{eq:uppbound} holds with $[0,R)^d$ replaced by $[0,8R)^d$. %{\color{red} A couple of 2 should be 4 below, no?} 
Moreover, by stationarity of the Poisson point process we may assume that there exists a random radius $r_*<\infty$ a.~s. such that for any dyadic $R \ge r_*$ \eqref{eq:uppbound} holds in the form
\begin{equation}\label{eq:step1uppbou2}
\frac1{R^d} \tilde{W}_{[-4R,4R)^d;p}^p (\mu, \tilde n_\mu) \le C\begin{cases}  \ln^\frac p2 R & \text{if $d=2$},\\
1 & \text{if $d\ge3$},
\end{cases}
\end{equation}
where $\tilde n_\mu = \frac{\mu([-4R,4R)^d)}{(8R)^d}$ is the (random) number density. Arguing in the same manner for $\nu$ we can deduce (possibly enlarging $r_*$) that
\begin{equation}\label{eq:step1uppbou2nu}
\frac1{R^d} \tilde{W}_{[-4R,4R)^d;p}^p (\nu, \tilde n_\nu) \le C\begin{cases}  \ln^\frac p2 R & \text{if $d=2$},\\
1 & \text{if $d\ge3$},
\end{cases}
\end{equation}
where $\tilde n_\nu = \frac{\nu([-4R,4R)^d)}{(8R)^d}$ is the (random) number density. By the restriction property of Lemma \ref{lem:localization} we may deduce that there exists a radius $R'\sim R$ such that 
\begin{equation}\label{eq:upplocwas}
\begin{split}
\lefteqn{W_{p,B_{R'}}^p(\mu, n_\mu) + W_{p,B_{R'}}^p(\nu, n_\nu)  =  \tilde{W}_{[-4R,4R)^d; p,B_{R'}}^p(\mu,  n_\mu) + \tilde{W}_{[-4R,4R)^d; p,B_{R'}}^p(\nu, n_\nu)} \\
& \lesssim \tilde{W}_{[-4R,4R)^d;p}^p (\mu, \tilde n_\mu) + \frac{R^p}{\tilde n_\mu}(\tilde n_\mu -1)^p + \tilde{W}_{[-4R,4R)^d;p}^p (\nu, \tilde n_\nu) + \frac{R^p}{\tilde n_\nu}(\tilde n_\nu -1)^p.
\end{split}
\end{equation}
Moreover, by the same argument as in {\sc Step 1} we may assume that \eqref{eq:step2uppbou} holds with $n_\mu$ and $n_\nu$ replaced by $\tilde n_\mu$ and $\tilde n_\nu$. Combining the latter with \eqref{eq:upplocwas} and \eqref{eq:step1uppbou2} and relabeling $R'$ yields \eqref{eq:step1uppbou}.

\medskip
{\sc Step 3}. Conclusion. Combining \eqref{eq:step2uppbou}, \eqref{eq:step2uppbounu} and \eqref{eq:step1uppbou} yields \eqref{eq:thmuppdata}. 
\end{proof}

Let us now turn to Lemma \ref{lem:uppbou}. In view of the dynamical formulation of optimal transport we  investigate the Moser coupling \cite[Appendix p. 16]{VilOandN} between $\mu$ and $n_\mu$.
%Let $\{X\}$ denote the Poisson point process on the torus $[0,L)^2$ of unit
%intensity; we consider the corresponding measure
%
%\begin{align}\label{ao08}
%\mu=\sum_{X}\delta_X.
%\end{align}
%
Let $(\cdot)_1$ denote a mollification on scale $1$, 
say the convolution\footnote{in the torus $[0,R)^d$}
with the standard Gaussian. Let $\phi$ denote the solution of the Poisson
problem on $[0,R)^d$
\begin{align}\label{ao03}
-\Delta\phi=\mu_1-\fint_{[0,R)^d}\mu_1= \mu_1 - n.
\end{align}
We are interested in the spatially averaged $p$-moment of its gradient
\begin{align}\label{ao07}
F:=\fint_{[0,R)^d}\frac{1}{p}|\nabla\phi|^p\quad\mbox{for any}\quad p<\infty.
\end{align}
We shall establish that thanks to the spatial averaging, $F$ has good concentration
properties around its expectation $\mathbb{E}F$. As we establish, the latter
is $O(\ln^\frac{p}{2}R)$ for $d=2$ and $O(1)$ for $d \ge 3$. What matters to us is that the probability 
that $F\gg\ln^\frac{p}{2}R$ if $d=2$ and $F \gg 1$ if $d\ge3$ is very small. 

\begin{lemma}\label{lem:1} There exists a constant $C<\infty$ such that for $R\ge C$ and any $p \ge 1$,
\begin{align}\label{ao00}
\mathbb{P}(F\ge C\ln^\frac{p}{2}R)\le \frac{C}{\ln^2R} && \text{if $d=2$}.
\end{align}
and
\begin{align}\label{eq:concFd>3}
\mathbb{P}(F\ge C)\le \frac{C}{R^{2d -4}} && \text{if $d\ge 3$}.
\end{align}
\end{lemma}

An inspection of the proof reveals that the exponent $2$  if $d=2$,  $2d-4$ if $d\ge 3$, on the r.~h.~s.~could be
replaced by any exponent $<\infty$ (on which $C$ will depend). 
However, it is sufficient for our purposes 
that the r.~h.~s.~of (\ref{ao00}) and \eqref{eq:concFd>3} is summable over dyadic $R$, which holds for
any exponent $>1$.
\medskip

The proof of Lemma \ref{lem:uppbou} is a direct consequence of Lemma \ref{lem:1}.

\begin{proof}[Proof of Lemma \ref{lem:uppbou}]
{\sc Step 1.} Definition of $r_*$. By Lemma \ref{lem:1} and a Borel-Cantelli argument we can deduce that there exists a constant $C$ and a random radius $r_* < \infty$ a.~s. such that for any dyadic radii $R \ge r_*$
\begin{equation}\label{eq:uppBenBre}
\int_{[0,R)^d} |\nabla \phi|^p \le C R^d \begin{cases}
\ln^\frac p2 R & \text{if $d=2$},\\
1 & \text{if $d \ge 3$},
\end{cases}
\end{equation}
where $\phi$ solves \eqref{ao03}. Moreover, arguing as for \eqref{eq:step2uppbou} we may assume that $r_*$ is large enough so that 
\begin{equation}\label{eq:numbdens}
\frac{\mu([0,R)^d)}{R^d} \in \bigg[ \frac12, 2 \bigg] \quad \text{for $R \ge r_*$}.
\end{equation}
%Indeed, by the Cram\'er-Chernoff's bounds we have that 
%\[
%\mathbb{P}\bigg( \frac{\mu([0,R)^2)}{R^2} \notin \bigg[\frac12, 2\bigg] \bigg) \lesssim \exp(-CR^2) 
%\]
%and thus \eqref{eq:numbdens} by a Borel-Cantelli argument. 

\medskip
{\sc Step 2.} Proof of \eqref{eq:uppbound}. By the triangle inequality and the semigroup contraction property of the Wasserstein distance we may write
\begin{equation}\label{eq:uppboutriang}
\tilde{W}_{[0,R)^d;p}^p (\mu, n)  \lesssim \tilde{W}_{[0,R)^d;p}^p (\mu, \mu_1) + \tilde{W}_{[0,R)^d;p}^p (\mu_1, n)  \lesssim 1 + \tilde{W}_{[0,R)^d;p}^p (\mu_1, n).
\end{equation}
We now turn to estimate the second item of the r.~h.~s. of \eqref{eq:uppboutriang}. By \eqref{eq:numbdens} we have the lower bound $n \ge \frac12$ which together with \eqref{eq:uppboutriang} implies \eqref{eq:uppbound} by the inequalities
\begin{equation}\label{eq:uppboufin}
\tilde{W}_{[0,R)^d,p}^p (\mu_1, n) \le 2^{p-1} p^p \int_{[0,R)^d} |\nabla \phi|^p \stackrel{\eqref{eq:uppBenBre}}{\le} C R^d \begin{cases}
\ln^\frac p2 R & \text{if $d=2$},\\
1 & \text{if $d\ge 3$}. 
\end{cases}
\end{equation}
The first inequality in \eqref{eq:uppboufin} can be derived from the dynamical description of the transport distance, namely the Benamou-Brenier formulation of the transport distance
\begin{equation}\label{eq:benbrep}
\tilde{W}_{[0,R)^2;p}(\mu_1, n) = \min \bigg\{ \int_0^1 \int_{[0,R)^d} \frac{|j|^p}{\rho^{p-1}} \ : \ \partial_t \rho + \nabla \cdot j = 0, \rho_0 = \mu_1, \rho_1 = n \bigg\},
\end{equation}
where $\rho, j$ have to be understood as distributions on $[0,1] \times [0,R)^d$ (see for instance \cite[Theorem 5.28]{Santam}). Indeed, the couple $(\rho, j)$, with $\rho = (1-t)^p \mu_1 + (1-(1-t)^p)n $ and $j = - p (1-t)^{p-1}\nabla \phi$ is an admissible candidate for \eqref{eq:benbrep} and by \eqref{eq:numbdens} we have the lower bound $\rho \ge (1-t)^p n \ge \frac 12 (1-t)^p$. 
\end{proof}

\begin{proof}[Proof of Lemma \ref{lem:1}]
By Jensen's inequality we can restrict ourselves to the case $p\ge 2$. By Chebyshev's inequality, it is enough to establish 
\begin{align}
%\mathbb{E} F&\lesssim \ln^\frac{p}{2}R,\label{ao01}\\
\mathbb{E}(F-C\ln^\frac{p}{2}R)_+^4&\lesssim \ln^{2(p-1)}R \quad \text{if $d=2$},
\label{ao02}
\end{align}
and 
\begin{equation}\label{eq:ao02}
\mathbb{E}(F-C)_+^4 \lesssim
R^{4-2d} \quad \text{if $d\ge 3$},
\end{equation}
for some constant $C$. 
We start by ignoring the spatial averaging in (\ref{ao07}) by considering 
\begin{align*}
G=G(\mu):=\nabla\phi(X)\quad\mbox{for some fixed point}\;X
\end{align*}
%
%By stationarity (i.~e.~shift-invariance) of the ensemble and the shift covariance
%of $\nabla\phi$, for (\ref{ao01}) it is enough to consider $G=\nabla\phi(X)$ for some fixed 
%point $X$ and to 
and establish its concentration. In fact, we shall derive a
mixture of exponential and Gaussian concentration:
\begin{align}\label{ao01bis}
-\ln\mathbb{P}(|G-\mathbb{E}G|\ge M)\gtrsim
\left\{\begin{array}{cl}
\frac{M^2}{\ln R}&\mbox{for $d=2$ and}\;M \ll \ln R\\
M&\mbox{otherwise}\end{array}\right\},
\end{align}
which by the layer cake representation implies the $L^p$ estimate in probability
\begin{align}\label{ao01ter}
\mathbb{E}^\frac{1}{p}|G-\mathbb{E}G|^p\lesssim\begin{cases}
\ln^\frac{1}{2} R & \text{if $d=2$}\\
1 & \text{if $d\ge3$}.
\end{cases}
\end{align}
By invariance of the ensemble and covariance of $\phi$ under reflection w.~r.~t.~
to the $d$ Cartesian hyper-planes crossing $X$, we have $\mathbb{E}G=0$, so that (\ref{ao01ter})
sharpens to
\begin{align}\label{ao19}
\mathbb{E}|G|^p\lesssim\begin{cases}
\ln^\frac{p}{2} R & \text{if $d=2$}\\
1 & \text{if $d\ge3$}.
\end{cases}
\end{align}
%
%This yields (\ref{ao01}) by averaging in $X\in[0,L)^d$.

\medskip
We now turn to the argument for (\ref{ao01bis}).
The concentration principle in \cite[Proposition 3.1]{Wu00} already
applied in \cite[Lemma 2.5]{HMO21} monitors
the change $D_{X_0}G:= G(\mu+ \delta_{X_0}) - G(\mu)$ of $G$ arising from adding a point at position $X_0$
to the point cluster. In view of (\ref{ao03}), the effect is given by
\begin{align}\label{ao14}
D_{X_0}G=\nabla D_{X_0}\phi(X)\quad\mbox{where}
\quad-\Delta D_{X_0}\phi=(\delta_{X_0})_1-R^{-d}.
\end{align}
We learn that $D_{X_0}\phi$ is the mollified potential function for a
periodic charge distribution at $X_0+(R\mathbb{Z})^d$ with a constant background charge
making the distribution overall neutral. This object is well-defined
on the level of its gradient and satisfies\footnote{we denote by ${\rm dist}(x,y+(R\mathbb{Z})^d) := \min_{k \in (R\mathbb{Z})^d} |x-y-k| $ the periodic distance on the torus $[0,R)^d$}
\begin{align}\label{ao15}
|D_{X_0}G|=|\nabla D_{X_0}\phi(X)|\lesssim
\big({\rm dist}(X,X_0+(R\mathbb{Z})^d)+1\big)^{1-d},
\end{align}
from which we learn, in the notation of \cite[Proposition 3.1]{Wu00},
\begin{align*}
\beta:=\sup_{X_0}|D_{X_0}G|\lesssim 1\quad\mbox{and}\quad
\alpha^2:=\int_{[0,R)^d}dX_0|D_{X_0}G|^2\lesssim \begin{cases}
\ln R & \text{if $d=2$},\\
1 & \text{if $d \ge 3$}.
\end{cases}
\end{align*}
This implies
\begin{align*}
\mathbb{P}(|G-\mathbb{E}G|\ge M)
\le 2\exp\big(-\frac{M}{2\beta}\ln(1+\frac{\beta M}{\alpha^2})\big),
\end{align*}
which is easily seen to imply (\ref{ao01bis}). 

\medskip

It is convenient to use a different concentration principle for (\ref{ao02}).
While for (\ref{ao01ter}), we used the concentration principle for the
``grand canonical ensemble'' of the Poisson point process on $[0,R)^2$, for (\ref{ao02})
is convenient to disintegrate this grand canonical ensemble into the ``canonical ensemble''
$\mathbb{E}_N$ of $N$ i.~i.~d.~points uniformly distributed.
Note that %(\ref{ao08}) and 
\eqref{ao03} assumes the form
\begin{align}\label{ao04}
-\Delta\phi=\mu_1-N\quad\mbox{on}\;[0,R)^d\quad
\mbox{where}\quad\mu=\sum_{n=1}^N\delta_{X_n},
\end{align}
where also the convolution refers to the torus $[0,R)^d$. The advantage is that we have an easy spectral gap estimate, which is
in fact just the tensorization of the standard Poincar\'e inequality with mean value zero
on $[0,R)^d$: For any suitable function $F=F(X_1,\cdots,X_N)$ we have
\begin{align*}
\mathbb{E}_N(F-\mathbb{E}_NF)^2\lesssim R^{2}\mathbb{E}_N\sum_{n=1}^N|\nabla_{X_n}F|^2.
\end{align*}
Applying this inequality with $F^2$ playing the role $F$, and appealing to
the Cauchy-Schwarz inequality in probability, we may upgrade this
standard version to the exponent $4$, which will be sufficient
for our purposes:
\begin{align*}
\mathbb{E}_N(F-\mathbb{E}_NF)^4\lesssim R^{4}\mathbb{E}_N\big(\sum_{n=1}^N|\nabla_{X_n}F|^2\big)^2.
\end{align*}
By the Cauchy-Schwarz inequality in $N$, and for our $F$ that is invariant under
permuting its argument, this yields
\begin{align}\label{ao11}
\mathbb{E}_N(F-\mathbb{E}_NF)^4\lesssim R^{4}N^2\mathbb{E}_N|\nabla_{X_N}F|^4.
\end{align}

\medskip

We now derive a suitable representation for $\nabla_{X_n}F$. 
From (\ref{ao07}) we obtain for the partial derivative $\nabla_{X_n}F$ for our $F$
%
%\begin{align*}
%\nabla_{X_n}F=\fint_{[0,L)^d}|\nabla\phi|^{p-2}%\nabla\phi\cdot\nabla\nabla_{X_n}\phi,
%\end{align*}
%
%where according to %(\ref{ao08}) and 
%(\ref{ao04}), $\nabla\nabla_{X_n}\phi$ is determined by
%
%\begin{align*}
%-\Delta\nabla_{X_n}\phi=(\nabla\delta_{X_n})_1\quad\mbox{on}\;[0,L)^d.
%\end{align*}
%
%Interpreting $\nabla\phi$ as periodic functions on $\mathbb{R}^d$,
%we may rewrite this as \cm{@@@Interpretation of the integral@@@}
%
\begin{align}
\nabla_{X_n}F=R^{-d}\int_{[0,R)^d}|\nabla\phi|^{p-2}
\nabla\phi\cdot(\nabla\nabla_{X_n}\bar\phi)_1,\label{ao09}\\
\mbox{where}\quad
-\Delta\nabla_{X_n}\bar\phi=\nabla\delta_{X_n}\quad\mbox{on}\;[0,R)^d.\nonumber
\end{align}
From the latter, we learn that $-\nabla\nabla_{X_n}\bar\phi(x)$ is the translation-invariant kernel, evaluated at $x-X_n$, of the Helmholtz projection, i.~e.~the $L^2([0,R)^d)$-orthogonal projection onto gradient fields (see \cite[Theorem 2.4.9]{hodge} and the discussion below). Since the latter operator is symmetric, and commutes with mollification, 
(\ref{ao09}) can be reformulated as
\begin{align}\label{ao10}
\nabla_{X_n}F=R^{-d}\nabla u_1(X_n),
\end{align}
where $-\nabla u$ is the Helmholtz projection of $|\nabla\phi|^{p-2}\nabla\phi$, that is
\begin{align}\label{ao13}
-\Delta u=\nabla\cdot|\nabla\phi|^{p-2}\nabla\phi.
\end{align}
Inserting (\ref{ao10}) into (\ref{ao11}), we obtain %(for our $d=2$)
\begin{align}\label{ao11bis}
\mathbb{E}_N(F-\mathbb{E}_NF)^4\lesssim \bigg(\frac{N}{R^d}\bigg)^2 R^{4-2d} \mathbb{E}_N|\nabla u_1(X_N)|^4.
\end{align}
%
%noting that under $\mathbb{E}$, that is, the Poisson distribution of $N$ with mean $L^d$,
%the first r.~h.~s.~factor is very close to one; it will be sufficient to control
%the forth moment:
%%
%\begin{align}\label{ao21}
%\mathbb{E}(\frac{N}{L^d})^4\lesssim 1.
%\end{align}

\medskip

In view of the Calder\'on-Zygmund estimate (see \cite[Chapter I, II and III]{stein} for a reference on classical Calder\'on-Zygmund's theory) for (\ref{ao13})
\begin{align}\label{ao12}
\fint_{[0,R)^d}|\nabla u|^4\le\fint_{[0,R)^d}|\nabla\phi|^{4(p-1)},
\end{align}
the plan now is to pass from $\mathbb{E}_N|\nabla u_1(X_N)|^4$ to
$\mathbb{E}_N\fint_{[0,R)^d}|\nabla u_1|^4$.
To this purpose, we will consider $\phi'$ defined in (\ref{ao04})
with $N$ replaced by $N-1$, and $u'$ defined like $u$ in (\ref{ao13})
with $\phi$ replaced by $\phi'$. Hence provided we can
\begin{align}\label{ao20}
\mbox{estimate}\quad\mathbb{E}_N|\nabla u_1(X_N)|^4\quad\mbox{by}\quad
\mathbb{E}_{N-1}|\nabla u_1'(X_N)|^4,
\end{align}
we may proceed to capitalize on the (stochastically) independence of
$u'$ of the uniformly distributed $X_N$ to the effect of
\begin{align*}
\mathbb{E}_N|\nabla u_1'(X_N)|^4=\mathbb{E}_{N-1}\fint_{[0,R)^d}|\nabla u_1'|^4.
\end{align*}
Using that $(\cdot)_1$ contracts the norm and (\ref{ao12}) we obtain by shift-covariance
of $\nabla\phi'$
\begin{align}\label{ao17}
\mathbb{E}_N|\nabla u_1'(X_N)|^4\lesssim\mathbb{E}_{N-1}|\nabla \phi'|^{4(p-1)}.
\end{align}
Hence provided we may 
\begin{align}\label{ao20bis}
\mbox{estimate}\quad\mathbb{E}_{N-1}|\nabla\phi'|^{4(p-1)}
\quad\mbox{by}\quad\mathbb{E}_N|\nabla\phi|^{4(p-1)}
\end{align}
we hope to obtain from (\ref{ao11bis}) that
\begin{align}\label{ao22}
\mathbb{E}_N(F-\mathbb{E}_NF)^4 \lesssim\bigg(\frac{N}{R^d}\bigg)^2 R^{4-2d}(\mathbb{E}_N|\nabla\phi|^{4(p-1)}+1).
\end{align}
Applying $\mathbb{E}$ to (\ref{ao22}), which just means applying the Poisson distribution
with mean $R^d$ to $N$, and using the Cauchy-Schwarz inequality on the latter,
and its good concentration property (on the level of the fourth moment), we obtain
\begin{align*}
\mathbb{E}(F-\mathbb{E}_NF)^4\lesssim R^{4-2d} \big(\mathbb{E}|\nabla\phi|^{8(p-1)}+1\big)^\frac{1}{2}.
\end{align*}
Inserting (\ref{ao19}) (with $p$ replaced by $8(p-1)$) we obtain 
\begin{align}\label{ao23}
\mathbb{E}(F-\mathbb{E}_NF)^4\lesssim
\begin{cases}
\ln^{2(p-1)}R & \text{if $d=2$},\\
R^{4-2d} & \text{if $d\ge3$},
\end{cases}
\end{align}
the major step towards (\ref{ao02}).

\medskip

We now turn to (\ref{ao20}) and (\ref{ao20bis}). 
%We now control $\nabla u-\nabla u'$ and $\nabla \phi-\nabla \phi'$. 
Momentarily introducing $j(z):=|z|^{p-2}z$ we note that
\begin{align*}
|j(z)-j(z')|\lesssim|z'|^{p-2}|z-z'|+|z-z'|^{p-1}.
\end{align*}
From (\ref{ao13}) we deduce a representation of $\nabla(u-u')$ as the Helmholtz
projection of $j(\nabla\phi)-j(\nabla\phi')$ on $[0,R)^d$. In view of the mollification
we obtain
\begin{align*}
|\nabla(u-u')_1(X_N)|\lesssim\int_{[0,R)^d}\big({\rm dist}(\cdot,X_N+(R\mathbb{Z})^d)+1\big)^{-d}
|j(\nabla\phi)-j(\nabla\phi')|.
\end{align*}
From (\ref{ao04}) we obtain, cf.~(\ref{ao15}),
\begin{align*}
|\nabla(\phi-\phi')|\lesssim\big({\rm dist}(\cdot,X_N+(R\mathbb{Z})^d)+1\big)^{1-d}\lesssim 1.
\end{align*}
The combination of these yields (using $p-1\ge 1$) 
\begin{align*}
\lefteqn{|\nabla (u-u')_1(X_N)|}\nonumber\\
&\lesssim\int_{[0,R)^d}\big({\rm dist}(\cdot,X_N+(R\mathbb{Z})^d)+1\big)^{1-2d}
(|\nabla\phi'|^{p-2}+1).
\end{align*}
Since by $1-2d<-d$ we have
\begin{align}\label{ao16}
\int_{[0,R)^d}\big({\rm dist}(\cdot,X_N+(R\mathbb{Z})^d)+1\big)^{1-2d}\lesssim 1,
\end{align}
this implies
\begin{align*}
\lefteqn{|\nabla (u-u')_1(X_N)|^4}\nonumber\\
&\lesssim\int_{[0,R)^d}\big({\rm dist}(\cdot,X_N+(R\mathbb{Z})^d)+1\big)^{1-2d}
(|\nabla\phi'|^{4(p-2)}+1).
\end{align*}
Applying $\mathbb{E}_N$, using the shift covariance on the level $\mathbb{E}_{N-1}$, 
and once more (\ref{ao16}) gives
\begin{align}\label{ao24}
\mathbb{E}_N|\nabla (u-u')_1(X_N)|^4\lesssim\mathbb{E}_{N-1}|\nabla\phi'|^{4(p-2)}+1.
\end{align}

\medskip
By the triangle inequality, (\ref{ao24}) deals with (\ref{ao20}); 
by \eqref{ao19} the additional error term is %already present on the r.~h.~s.~of (\ref{ao17}),
of higher order of the one already present on the r.~h.~s.~of (\ref{ao17}), and the $+1$ is of higher order.  The argument for (\ref{ao20bis}) is easier.
In fact, for later purpose, we shall establish the monotonicity
\begin{align}\label{ao25}
f(N):=\mathbb{E}_N\frac{1}{p}|\nabla\phi|^p=\mathbb{E}_NF\quad\mbox{satisfies}\quad f(N-1)\le f(N).
\end{align}
Appealing to the monotonicity with $p$ replaced by $4(p-1)$ we obtain (\ref{ao20bis}).
Here comes the argument for (\ref{ao25}). By convexity of $z\mapsto\frac{1}{p}|z|^p$ we have
$\frac{1}{p}|z|^p$ $\ge\frac{1}{p}|z'|^p$ $+|z'|^{p-2}z'\cdot (z-z')$,
which we use in form of
\begin{align*}
\frac{1}{p}|\nabla\phi|^p\ge\frac{1}{p}|\nabla\phi'|^p+|\nabla\phi'|^{p-2}\nabla\phi'
\cdot\nabla(\phi-\phi').
\end{align*}
Since $\nabla\phi'$ is independent of $\nabla(\phi-\phi')$, and since the expectation
of the latter vanishes as we discussed above based on reflection symmetry, 
this implies (\ref{ao25}). 
Hence we have completed the argument for (\ref{ao22}) and thus (\ref{ao23}).

\medskip

It remains to post-process (\ref{ao23}) to (\ref{ao02}). To this purpose, we
prove a partial reverse of (\ref{ao25}), namely
\begin{align}\label{ao26}
f(N)\lesssim f(N')\quad\mbox{provided}\;N\le 2N'.
\end{align}
Indeed, now we start from $\frac{1}{p}|z|^p$ $\le\frac{1}{p}|z'|^p$ $+|z|^{p-2}z\cdot (z-z')$,
in form of
\begin{align*}
\frac{1}{p}|\nabla\phi|^p\le\frac{1}{p}|\nabla\phi'|^p+|\nabla\phi|^{p-1}|\nabla(\phi-\phi')|.
\end{align*}
We now apply H\"older's inequality in probability to the effect of
\begin{align*}
f(N)&\le f(N')+(pf(N))^\frac{p-1}{p}(pf(N-N'))^\frac{1}{p}\nonumber\\
&\stackrel{(\ref{ao25})}{\le}f(N')+(pf(N))^\frac{p-1}{p}(pf(N'))^\frac{1}{p}
\quad\mbox{provided}\;N\le 2N',
\end{align*}
so that (\ref{ao26}) follows from Young's inequality.

\medskip

Equipped with (\ref{ao25}) and (\ref{ao26}) we now may pass from
(\ref{ao23}) to (\ref{ao02}), where we use that by (\ref{ao19}) and shift-covariance
of $|\nabla\phi|^p$ we have
\begin{align}\label{ao27}
\mathbb{E}F\le\mathbb{E}^\frac{1}{8}F^8\lesssim
\begin{cases}
\ln^\frac{p}{2}R & \text{if $d=2$},\\
1 & \text{if $d \ge 3$}.
\end{cases}
\end{align}
We fix an $N_0\in\mathbb{N}$ with
\begin{align*}
N_0\approx 2R^d,
\end{align*}
so that by the concentration properties of the Poisson distribution
we have thanks to the factor of $2$
\begin{align}\label{ao28}
\mathbb{P}(N\ge N_0)\quad\mbox{is sub-algebraic in}\;R\;
\mbox{while}\quad\mathbb{P}\bigg(N\le\frac{N_0}{2}\bigg)\sim 1.
\end{align}
By (\ref{ao27}), the first item in (\ref{ao28}) transmits to
\begin{align}\label{ao28bis}
\lefteqn{\mathbb{E}I(N\ge N_0)(F-f(N_0))_+^4}\\
&\le\big(\mathbb{P}(N\ge N_0)\mathbb{E}F^8\big)^\frac{1}{2}
\quad\mbox{is sub-algebraic in}\;R.\nonumber
\end{align}
Once more by (\ref{ao27}), the second item in (\ref{ao28}) implies by Chebyshev
\begin{align}\label{ao30}
f(N_0)\stackrel{(\ref{ao26})}{\lesssim} f\bigg(\frac{N_0}{2}\bigg)\lesssim\mathbb{E}[f(N)]=\mathbb{E}F
\lesssim
\begin{cases}
\ln^\frac{p}{2}R & \text{if $d=2$}\\
1 & \text{if $d\ge3$}.
\end{cases}
\end{align}
Now for the complementary portion to (\ref{ao28bis}), we may appeal to (\ref{ao25}) 
in order to connect to (\ref{ao23}):
\begin{align}\label{ao29}
\mathbb{E}I(N\le N_0)(F-f(N_0))_+^4\le\mathbb{E}(F-f(N))^4\lesssim
\begin{cases}
\ln^{2(p-1)}R & \text{if $d=2$},\\
R^{4-2d} & \text{if $d\ge3$}.
\end{cases}
\end{align}
The desired (\ref{ao02}) now follows from combining (\ref{ao28bis}) with (\ref{ao29})
and inserting (\ref{ao30}).
\end{proof}

\subsection{A \texorpdfstring{$L^\infty$}{L infinity}-estimate}
In this section we improve the $L^0$ estimate \cite[Lemma 2.1]{HMO21} into a $L^{\infty}$ estimate for the displacement $|T(X)-X|$. The proof of Lemma \ref{lem:Linfty} is in the spirit of \cite[Lemma 3.1]{koch23}, see also \cite{Gut22} for an $L^\infty$ estimate in the regime $p \ge 2$ using different techniques. 
\begin{lemma}\label{lem:Linfty}
For every $\eps >0$ there exists a random radius $r_*<\infty$ a.~s.  such that for every $R \geq r_*$
 \begin{equation}\label{eq:Linfty}
   \abs{T \bra{X} - X} \leq  \eps R \quad \mbox{provided that} \; X \in (-R,R)^d.
 \end{equation}
\end{lemma}

The proof is very similar to \cite[Lemma 2.2]{HMO21}. As opposed to the quadratic case the support of $T$ is not monotonic for general $p\ge 1$. To overcome this additional difficulty we need to argue by $p$-cyclically monotonicity to exploit the geometry of the support of the matching. To be more precise we consider a toy case. Let $0$ denote the origin in $\R^d$ and let $e_1=(1,0,\dots,0)$. We now think of the origin playing the role of a Poisson point, and $e_1$ being a Poisson point, close to $0$, which is transported by a moderate distance, in particular we may suppose $T(e_1) = e_1$. By $p$-cyclically monotonicity we may write
\[
|T(0)|^p \le |T(0) - e_1|^p + 1.
\]
The latter defines a constraint for $T(0)$ in which the origin is transported. Our aim is to understand the not admissible set in which $0$ is transported. In particular, we show in the next lemma that the region in which the origin is not transported contains a convex set, i.~e. a cone. 

\begin{lemma}\label{lem:conebarrier}
Let $1 < p < \infty$ and $\alpha \in (0, \pi)$. Define the domain $U \subset \R^d$ 
\[
U:= \{ x = (x_1, x') \in \mathbb{R}\times \R^{d-1} \ | \ |x|^p \ge ((x_1 - 1)^2 + |x'|^2)^\frac{p}{2} + 2 \}. 
\]
Then there exists a cone with vertex $v = (\bar x(p,\alpha), 0)$, aperture $\alpha$ and axis $(1,0,\dots,0)$ which lies in $D$. 
\end{lemma}

\begin{proof}
We denote $y= | x'|$ and introduce the function 
$
F(x,y) := (x^2 + y^2)^\frac p 2 - ((x-1)^2 + y^2)^\frac p 2.
$

\medskip
{\sc Case $p \ge 2$}. We claim that there exists a constant $c_0 > 0$ such that for $c\ge c_0$ the half-space
\begin{equation}\label{eq:subsetpge2}
C_p = \cur{x \ \colon \ x_1 \ge c} \quad \text{is contained in $U$}. 
\end{equation}
We start by noticing that by a direct calculation since $p \ge 2$, for $x \in C_p$ and $c \ge \frac12$ we have $\partial_y F(x,y) \ge 0$. Thus for $c \ge \frac12$, in order to ensure \eqref{eq:subsetpge2} it is enough to show $F(x_1, 0) \ge 2.$
We note that for $x \in C_p$ and $c \ge 1$, the inequality $F(x_1,0) \ge 2$ is satisfied if the following holds true
\[
\bra{1 + \frac 1 {c-1}}^p - 1 \ge \frac 2 {(c-1)^p}.
\]
By Bernoulli's inequality, (i.~e. $(1+u)^n \ge 1 + nu$ for $n \ge 0, u > -1$) the latter is satisfied if
$
c \ge 1 + \bra{\frac 2 p}^{\frac 1 {p-1}}.
$
Thus, choosing $c_0 = 1 + \bra{\frac 2 p}^{\frac 1 {p-1}}$ yields \eqref{eq:subsetpge2}. 

\medskip
{\sc Case $1<p\le 2$}. We claim that there exists a constant $c_0 > 0$ such that for $c\ge c_0$ the cone
\begin{equation}\label{eq:subsetple2}
C_p = \cur{x \ \colon \ y \le \alpha (x_1 - c) } \quad \text{is contained in $U$}.
\end{equation}
We start by noticing that by a direct calculation since $1< p \le 2$, for $x \in C_p$ and $c\ge \frac12$ we have $\partial_y F(x,y) \le 0$. 
%\textcolor{blue}{
%Indeed,
%\[
%\partial_y F(x,y) = p y \cur{ (x^2 + y^2)^{\frac{p}{2}-1} - ((x-1)^2 + y^2)^{\frac p 2 - 1}} \ge 0 \Leftrightarrow x \le \frac12.
%\]
%Since $x \in C_p, x_1 \ge \frac y \alpha + c$, thus if $c \ge \frac12$ for $x \in C_p$, $\partial_y F(x,y) \le 0$. 
%} 
Thus it suffices to show that for $x_1 \ge c$ it holds
$
F(x_1, \alpha (x_1 - c)) \ge 2. 
$
Let us denote by $g (z)$ the function $g(z) := (z^2 + \alpha^2 (x_1 - c)^2)^\frac p2$. By the mean value theorem, %for $c \ge 2 (1+ \alpha^2)$, 
there is $\xi \in [x_1 - 1, x_1]$ such that, for $1 < p \le 2$
\[
\begin{split}
F(x_1, \alpha(x_1 -c)) & = g(x_1) - g(x_1 - 1) \\
& = p \frac{\xi}{(\xi^2 + \alpha^2(x_1 - c)^2)^\frac{2-p}2} \\
& \ge p \frac{\xi}{(\xi^{2-p}+ \alpha^{2-p}(x_1 - c)^{2-p})} \\
& \ge p \frac{x_1 - 1}{(1+\alpha^{2-p})x_1^{2-p}  + \alpha^{2-p} c^{2-p}} =: h(x_1).
\end{split}
\]
Note that since $1<p\le2$ we have that $h(x_1) \rightarrow \infty$ for $x_1 \rightarrow \infty$ and is increasing for $x_1$ sufficiently large, hence we can choose $c_0 < \infty$ such that $h(c_0)\ge 2$. Finally, noting that for $x \in C_p$ we have $x_1 \ge \frac y \alpha + c \ge c_0$ yields \eqref{eq:subsetple2}.
\end{proof}

\begin{proof}[Proof of Lemma \ref{lem:Linfty}]

\medskip

{\sc Step 1}. Definition of $r_*=r_*(\eps)$ given $0<\eps\ll 1$ 
as the maximum of three $r_*$'s.
First, by \cite[Lemma 2.1]{HMO21}, there exists a (deterministic) length $L<\infty$ and 
the (random) length $r_*<\infty$ such that for $4R\ge r_*$,
the number density of the Poisson points in $(-2R,2R)^d$ 
transported further than the ``moderate distance'' $L$ is small in the sense of
\begin{align}\label{fw17}
\#\{\,X\in(-2R,2R)^d\,|\,|T(X)-X|>L\,\}\le(\eps 4R)^d.
\end{align}
Second, by Lemma \ref{lem:uppdata} we may also assume that $r_*$ is so large that for $R\ge r_*$,
the non-dimensionalized transportation distance of $\mu$ to its number density $n$
is small, and that by the concentration properties of the Poisson point process $n\approx 1$, in the sense of
\begin{align}\label{fw11}
W^p_{p,(-2R,2R)^d}(\mu,n)+\frac{(4R)^{d+p}}{n}(n-1)^p\le
(\eps 4 R)^{d+p}.
\end{align}
Third, w.~l.~o.~g.~we may assume that $r_*$ is so large that 
\begin{align}\label{fw14}
L\le\eps r_*.
\end{align}
We now fix a realization and $R\ge r_*$.

\medskip

{\sc Step 2}.

There are enough Poisson points on mesoscopic scales.
We claim that for any cube $Q\subset(-2R,2R)^d$ of ``mesoscopic'' side length 
\begin{align}\label{fw12}
r\gg\eps R
\end{align}
we have
%\footnote{The notation $A \gtrsim B$ means that there exists a constant $C>0$, which may only depend on $d$, such that $A \geq C B$. Similarly $A \lesssim B$ means that there exists a constant $C>0$, which may only depend on $d$, such that $A \leq C B$.}
%
\begin{align}\label{fw13}
\#\{\,X\in Q\,\}\gtrsim r^d. 
\end{align}
Indeed, it follows from the definition of $W_{p,(-2R,2R)^d}(\mu,n)$ that for any
Lipschitz function $\eta$ with support in $Q$ we have
\begin{align*}
\big|\int\eta d\mu-\int\eta ndy\big|\le({\rm Lip}\eta)\big(\int_Qd\mu+n|Q|\big)^\frac{1}{p'}
W_{p, (-2R,2R)^d}(\mu,n).%\ge n\int_{(-2R,2R)^d}\eta,
\end{align*}
where $\frac1p + \frac1{p'}=1$. Indeed, by H\"older's inequality 
\[
\begin{split}
\big|\int\eta d\mu-\int\eta ndy\big| & = \big| \int \bra{\eta\bra{x} - \eta \bra{y}} d\pi \bra{x,y} \big| \\
& \leq \bra{ {\rm Lip} \eta} \int |x-y| d \pi \\
& \le ({\rm Lip}\eta)\big(\int_Qd\mu+n|Q|\big)^\frac{1}{p'}
W_{p, (-2R,2R)^d}(\mu,n).%\ge n\int_{(-2R,2R)^d}\eta
\end{split}
\]
We now specify to an $\eta\le 1$ supported in $Q$, to the effect of $\int\eta d\mu$
$\le\int d\mu$ $=\#\{\,X\in Q\,\}$, so that by Young's inequality and the trivial inequality $\bra{x+y}^\frac1p \le x^\frac1p + y^\frac1p$ for $x,y >0, p>1$ we have
\begin{align}\label{fw25}
\int\eta ndy&\lesssim
\#\{\,X\in Q\,\}+({\rm Lip}\eta)^{p}W^{p}_{p,(-2R,2R)^d}(\mu,n)\nonumber\\
&+({\rm Lip}\eta)\big(n|Q|\big)^\frac{1}{p'}
W_{p, (-2R,2R)^d}(\mu,n).%\ge n\int_{(-2R,2R)^d}\eta.
\end{align}
At the same time, we may ensure $\int_{(-2R,2R)^d}\eta\gtrsim r^d$ and 
${\rm Lip}\eta\lesssim r^{-1}$, 
so that by (\ref{fw11}), which in particular ensures $n\approx 1$, (\ref{fw25})
turns into
\begin{align*}
r^d\lesssim 
\#\{\,X\in Q\,\}+r^{-p}(\eps R)^{d+p}+r^{\frac{d}{p'}-1}(\eps R)^{\frac{d}{p}+1}.
\end{align*}
Thanks to assumption (\ref{fw12}) we obtain (\ref{fw13}).

\medskip
{\sc Step 3}.
Iteration. At mesoscopic distance around a given point $X\in(-R,R)^d$, 
there are sufficiently
many Poisson points that are transported only over a moderate distance in any direction. 
More precisely,  %we claim that provided (\ref{fw12}) holds, 
%\begin{equation}\label{eq:linftystep3}
%\begin{split}
%& \text{for any unitary vector $e$ there exists a Poisson point $X' \in B_{\sqrt{2}r} \bra{X + 2 r e}$} \\
%&  \text{such that $|T (X') - X'| \le L$.}
%\end{split}
%\end{equation}
we claim that for any cube
$Q\subset(-2R,2R)^d$ of side-length satisfying (\ref{fw12}) 
we have
\begin{align}\label{fw16}
\mbox{there exists}\;X\in Q\;\mbox{with}\;|T(X)-X|\le L.
\end{align}
We suppose that (\ref{fw16}) were violated for some cube $Q$. 
%In view of (\ref{fw14}),
%all Poisson points in $X$ would be transported a distance larger than $L$.
By (\ref{fw13}), there are $\gtrsim r^d$ of such points. By assumption
(\ref{fw12}), there are thus $\gg(\eps R)^d$ Poisson points in $(-2R,2R)^d$
that get transported by a distance $>L$, which contradicts (\ref{fw17}).
%We now turn to \eqref{eq:linftystep3}. If \eqref{eq:linftystep3} was violated $B_{\sqrt{2}r} \bra{X + 2 r e}$ contains a cube $Q$ of side length $r$ satisfying \eqref{fw17} such that 
%\[
%|T(X) - X| > L \;\mbox{for any}\;X \in Q,
%\]
%which contradicts \eqref{fw16}. 

\medskip

{\sc Step 4}. Building barriers. We show that for any Poisson point $X$ and any unitary vector $e$, if we are given a cube $Q^X$ with barycenter $X$ of side-length satisfying \eqref{fw12} and $X'\in Q+2r e$ such that
%$X'\in B_{\sqrt{2}r} \bra{X + 2 r e}$ such that 
\[
| T(X') - X'| \le L
\]
there exists a cone $C_{X,X'}$ with vertex $X + r \rho \frac{X'-X}{|X'-X|}$ for some finite constant $\rho = \rho (p,d) > 0$, aperture $1$ and axis $\frac{X'-X}{|X'-X|}$ such that $T(X) \notin C_{X,X'}$. Indeed, by $p$-cyclically monotonicity of $T$ we get
\[
 |T(X)-X|^p \le |T(X)-X|^p + |T(X')-X'|^p \le |T(X')-X|^p + |T(X)-X'|^p.
\]
By a change of coordinates we may assume that $X = 0$ and $X'=(\tau,0)$ with $\tau \in (\frac32 r, \frac{\sqrt{26}}2r)$. In particular, writing $T(X) = (y_0, y_1)$ with $y_0 \in \R$ and $y_1 \in \R^{d-1}$ we get
\[
|T(X)|^p \le (L + \tau)^p + (|y_0 - \tau|^2 + |y_1|^2)^\frac{p}2.
\]
Introduce $(\tilde{y}_0, \tilde{y}_1) = \frac1\tau (y_0,y_1)$. Then we have
\[
\begin{split}
|\tilde{y}|^p & \le (|\tilde{y}_0 - 1|^2 + |\tilde{y}_1|^2)^\frac{p}2 + \bra{\frac{L}{\tau} + 1}^p \\
& \stackrel{\eqref{fw14}, \eqref{fw12}}{\le} (|\tilde{y}_0 - 1|^2 + |\tilde{y}_1|^2)^\frac{p}2 + 2.
\end{split}
\]
In particular by Lemma \ref{lem:conebarrier} there exists (going back to the original coordinates) a cone $C_{X,X'}$ with vertex $X + r \rho \frac{X'-X}{|X'-X|}$, aperture $1$ and axis $\frac{X'-X}{|X'-X|}$ such that $T(X) \notin C_{X,X'}$.

\medskip
{\sc Step 5}.  All Poisson points are transported over distances $\ll R$.  We claim that for all Poisson points $X$ 
\begin{equation}\label{fw21}
|T(X) - X| \lesssim \eps R \quad \mbox{provided}\; X \in (-R,R)^d.
\end{equation}
Choosing $c(d,p)$ directions $e_i$, we get points $\cur{X_i}_{i=1}^{c(d)}$ with $X_i \in Q^{X_i} + 2 re_i$ for some finite constants $\rho_i >0$, cones $C_{X,X_i}$ with vertexes $X + r \rho_i \frac{X_i - X}{|X_i-X|}$, aperture $1$ and axes $\frac{X_i - X}{|X_i -X|}$ and a finite constant $\rho_i \le \bar{\rho} < \infty$ for every $i$ such that 
\[
T(X) \notin \bigcup_{i=1}^{c(d)} C_{X,X_i} \quad\mbox{and}\quad \R^d \setminus B_{\bar \rho r} (X) \subset \bigcup_{i=1}^{c(d)} C_{X,X_i}.
\]
In particular, 
\[
|T(X)-X| \le \bar \rho r \lesssim r.
\]
Since (\ref{fw12}) was the
only constraint on $r$, we obtain (\ref{fw21}).
\end{proof}

%%%%%%%%%%%%
%%%%%%%%%%%%
%%%%%%%%%%%%
%%%%%%%%%%%%
%%%%%%%%%%%%
%%%%%%%%%%%%
\subsection{Application of the \texorpdfstring{$p$}{p}-Harmonic Approximation Theorem}\label{sec:mainsteps}
%Let us denote by $E_p$ the $p$-energy 
%\[
%E_p (R) := \frac1{R^d} \sum_{X\in B_R\;\mbox{or}\;T(X)\in B_R} |T(X)-X|^p.
%\]

\begin{lemma}\label{lem:harmonicapprox}
Let $p>1$. There exist a constant $C$ and a random radius $r_* < \infty$ a.~s. such that for every $R\geq r_*$ we have
\begin{equation}\label{eq:L2statement}
\frac{1}{R^d}\sum_{X \in B_R \; \mbox{or} \; T \bra{X} \in B_R} |T\bra{X}  - X|^p  \leq  C \ln^{\frac p2} R.
\end{equation}
\end{lemma}
\begin{proof}
The proof relies on the $p$-harmonic approximation result \cite[Theorem 1.1]{koch23}. This result establishes that for any $0<\tau\ll 1$, there exists
an $\eps>0$, a constant $C>0$ (which does not depend on $\tau$) and a constant $C_\tau<\infty$ such that provided for some $R$
\begin{equation}\label{eq:hypharmapprox}
\frac1{R^p} E_p(4R) + \frac1{R^p} D_p(4R)\leq \eps
\end{equation}
there exists a $p$-harmonic gradient field $\nabla \Phi$ such that
\begin{equation}\label{eq:harmapproxoutput}
\frac{1}{R^d} \sum_{X\in B_R\;\mbox{or}\;T(X)\in B_R}  \left|T \bra{X} - X -|\nabla \Phi(X)|^{p'-2} \nabla \Phi(X)\right|^p \leq \tau E_p(4R) + C_\tau D_p(4R),
\end{equation}
and\footnote{Note that in \cite[(2.21)]{HMO21} the constants on the right hand side were both labeled $C_\tau$. However, the constant which controls $\sup_{B_{2R}} |\nabla \Phi|^2$ in \cite[(2.21)]{HMO21} does not depend on $\tau$. This is important in the proof of \eqref{eq:L2basestepdyadic}. To avoid confusion we labeled the constants in \eqref{eq:harmapproxoutput} and \eqref{eq:harmapproxoutput2} differently.}
\begin{equation}\label{eq:harmapproxoutput2}
\sup_{B_{2R}} |\nabla \Phi|^\frac p {p-1} \leq C \bra{E_p(4R) + D_p(4R)},
\end{equation}
where $p'$ is the conjugate exponent of $p$.
The fraction $\tau$ will be chosen at the end of the proof. %Note that in \eqref{eq:datatermp} $D_p \bra{R}$ is defined on boxes $(-R,R)^d$ while \cm{[Lukas' notes, p-Harmonic approximation]} \eqref{eq:hypharmapprox} requires balls. 
%Since $B_{6R} \subseteq (-6R, 6R)^d$ we may assume that \eqref{eq:hypharmapprox} holds also for $B_{6R'}$ for $R'$ close to $R$, see Lemma \ref{lem:localization}.

%%%%%%%%%%%%
%Recall $E(R)$ and $D(R)$ from \eqref{eq:energy} and \eqref{eq:data}. We want to apply the harmonic approximation theorem from \cite[Theorem 1.4]{GHO}. To this end, we need to ensure the main condition that for $\eps >0$ we have
%\begin{equation}\label{eq:hypharmapprox}
%\frac1{R^2} E(6R) + \frac1{R^2} D(6R)\leq \eps.
%\end{equation}
%By Poisson concentration (see Lemma \ref{lem:Poiconcentration} ) we can assume that for $R\geq r_*$ we have $\frac{\mu(B_R)}{R^d}\in \sqa{\frac12,2}.$ Fix $\eps>0$. 
\medskip
{\sc Step 1}. Definition of $r_*$ depending on $\tau$. For $0< \tau \ll 1$ let $\eps = \eps \bra{\tau}$ be as above. %Arguing as in {\sc Step 2} of the proof of \cite[Lemma 2.6]{HMO21} we can show that there exist a constant $C$ and a random radius $r_*< \infty$ a.~s. such that for any dyadic radii $R \ge r_*$
%\begin{equation}\label{eq:concnumbdens}
%\frac{R^p}{n}(n-1)^p \le C \ln^\frac p2 R. 
%\end{equation}
%Combining \eqref{eq:concnumbdens} with Lemma \ref{lem:uppbou} and \eqref{eq:datatermp} 
By Theorem \ref{lem:uppdata} we may assume that $r_*$ is large enough so that for a (random) sequence of approximately dyadic radii $R\geq r_*$
\begin{equation}\label{eq:L2upperdata}
D_p(R) \leq C \ln^\frac p2 R.
\end{equation}  
Moreover, by Theorem \ref{lem:uppdata} possibly enlarging $r_*$ and we may assume that for a (random) sequence of approximately dyadic radii\footnote{by the footnote of \cite[Theorem 1.1]{koch23} it suffices that \eqref{eq:hypharmapprox} is satisfied for a sequence of radii $4R'\sim 4R$, thus we do not rename the sequence in \eqref{eq:L2bounddata}}
\begin{equation}\label{eq:L2bounddata}
\frac{D_p (4R)}{R^p} \leq \frac{\eps}2 .
\end{equation}
Note that only the bound \eqref{eq:L2upperdata} is specific to $d=2$. Moreover, the estimate \eqref{eq:L2bounddata} is not sharp, but it is enough for our purpose. From now on, we restrict ourselves to the sequence of approximately dyadic radii $R$ coming from \eqref{eq:L2upperdata} and \eqref{eq:L2bounddata}, which we may do w.~l.~o.~g.~ for \eqref{eq:L2statement}. Note that by the bound on $D_p\bra{4R}$ in \eqref{eq:L2bounddata} and the second and fourth term in the definition of $D_p\bra{R}$ in \eqref{eq:datatermp}
\begin{equation}\label{eq:L2dataestpoint}
\# \bra{\cur{X \in B_R} \cup \cur{ T(X) \in B_R}} \leq C R^d.
\end{equation}
Moreover, we may assume that $r_*$ is large enough so that \eqref{eq:Linfty} holds. Since $B_{4R}\subset (-4R,4R)^d$ we may sum \eqref{eq:Linfty} over $B_R$ to obtain for $R\geq r_*$ 
\[
\frac{1}{(4R)^d}\sum_{X\in B_{4R}} |T \bra{X} - X|^p \leq \frac\eps4 R^p.
\]
By symmetry, potentially enlarging $r_*$, we may also assume that \eqref{eq:Linfty} holds with $X$ replaced by $T \bra{X}$ so that for $R \ge r_*$ both
\begin{equation}\label{eq:L2preimageingoodball}
T \bra{X} \in B_R \Rightarrow X \in B_{2R}
\end{equation}
and %{\color{red} why the $1/4$? Below you have $2/4=1/2$?}
\[
\frac{1}{(4R)^d}\sum_{T\bra{X}\in B_{4R}} |T \bra{X} - X|^p \leq \frac\eps4 R^p,
\]
thus
\begin{equation}\label{eq:energyboundharmapprox}
E_p \bra{4R} = \frac{1}{(4R)^d}\sum_{X\in B_{4R}\;\mbox{or}\;T(X)\in B_{4R}} |T \bra{X} - X|^p \leq \frac\eps2 R^p,
\end{equation}
and in particular \eqref{eq:hypharmapprox} holds.
Finally by \cite[Lemma 2.1]{HMO21} we may assume, possibly enlarging $r_*$, that there exists a deterministic constant $L_\tau$ and for $R\geq r_*$ we both have 
\begin{equation}\label{eq:L2ergodic}
\# \bra{\cur{X \in Q_R \ | \ \abs{T \bra{X} - X} > L_\tau} \cup \cur{T(X) \in Q_R \ | \ \abs{T \bra{X} - X} > L_\tau}} \leq \tau R^d.
\end{equation}
and
\begin{equation}\label{eq:L2LleqlnR}
L_\tau^p \leq \ln^\frac p2 R.
\end{equation}
%Finally, by symmetry we may assume, eventually enlarging $r_*$, that \eqref{eq:Linfty} holds with $T \bra{X}$ replaced by $X$. This in particular implies that
%\begin{equation}\label{eq:L2preimageingoodball}
%T \bra{X} \in B_R \Rightarrow X \in B_{\bra{1+\eps}R}.
%\end{equation}

\medskip
{\sc Step 2.} 
Application of harmonic approximation. For all $R \geq r_*$
\begin{equation}\label{eq:L2basestep}
E_p \bra{R} \leq {\tau} E_p \bra{32 R} + C_\tau \ln^\frac p2 R.
\end{equation}
We start by showing that for the (random) sequence of approximately dyadic radii of {\sc Step 1} it holds for $R \ge r_*$
\begin{equation}\label{eq:L2basestepdyadic}
E_p \bra{R} \leq {\tau} E_p \bra{4 R} + C_\tau \ln^\frac p2 R.
\end{equation}
We split the sum according to whether the transportation distance is moderate or large. On the latter we use the harmonic approximation:
%Since \eqref{eq:hypharmapprox} holds, we may apply the harmonic approximation to get a constant $C_\tau$ and a harmonic gradient field $\Phi$ such that
%$$ \frac{1}{|B_R|} \sum_{X\in B_R\;\mbox{or}\;T(X)\in B_R}  \left|T \bra{X} - X -\nabla \Phi(X)\right|^2 \leq \tau E(6R) + C_\tau D(6R)$$
%and
%$$R^2 \sup_{B_{2R}} |\nabla^2\Phi|^2 + \sup_{B_{2R}} |\nabla \Phi|^2 \leq C_\tau \bra{E(6R) + D(6R)}.$$
%Consider the set of Poisson points  which are transported by a distance distance greater than $L$.
%Hence, we can estimate
\begin{align*}
 \lefteqn{\frac1{R^{d}} \sum_{\bra{X\in B_R\;\mbox{or}\;T(X)\in B_R}\:\mbox{and}\; \abs{T\bra{X}-X}>L_\tau} |T\bra{X} - X|^p}\\
 & \leq  \frac{2^p}{R^{d}} \sum_{{X\in B_R\;\mbox{or}\;T(X)\in B_R}} |T \bra{X} -X -|\nabla \Phi(X)|^{p'-2}\nabla\Phi \bra{X}|^p  \\
 & + \frac{2^p}{R^{d}} \sum_{\bra{X\in B_R\;\mbox{or}\;T(X)\in B_R}\:\mbox{and}\; \abs{T\bra{X}-X}>L_\tau} |\nabla\Phi(X)|^{\frac p {p-1}}  \\
& \stackrel{\eqref{eq:L2ergodic}}{\leq} 2^p\bra{\tau E_p(4R) + C_\tau D_p(4R)} + 2^p\tau \sup_{B_R} |\nabla \Phi|^2\\
& \stackrel{\eqref{eq:harmapproxoutput},\eqref{eq:harmapproxoutput2}, \eqref{eq:L2preimageingoodball}}{\leq}  2^p \tau E_p(4R) + 2^p C_\tau D(4R) + 2^p\tau C (E_p(4R)+D_p(4R))\\
 & =  2^p \tau \bra{1 + C} E_p(4R) + 2^p \bra{C_\tau+\tau C} D_p(4R).
\end{align*}
%Note that by the bound on $D\bra{6R}$ in \eqref{eq:hypharmapprox} and the second and fourth term in the definition of $D\bra{R}$ in \eqref{eq:data}
%\begin{equation*}
%\# \bra{\cur{X \in B_R} \cup \cur{ T(X) \in B_R}} \leq C R^d.
%\end{equation*}
%Indeed, by definition \eqref{eq:data} and \eqref{eq:hypharmapprox} we may deduce
%\begin{align*}
%\# \cur{X \in B_R} \leq \bra{1 + \sqrt{\eps}} \abs{B_R} < 2 \abs{B_R} && \# \cur{T \bra{X} \in B_R} \leq \bra{1 + \sqrt{\eps}}\abs{B_R} < 2 \abs{B_R}.
%\end{align*}
The last estimate combines to
\begin{align*} 
\lefteqn{\frac1{R^d}\sum_{X\in B_R\;\mbox{or}\;T(X)\in B_R}|T \bra{X} - X|^p}\\  & =  \frac1{R^d} \sum_{\bra{X\in B_R\;\mbox{or}\;T(X)\in B_R}\:\mbox{and}\; \abs{T\bra{X}-X}\leq L_\tau} |T \bra{X} - X|^p \\
& +  \frac1{R^d}\sum_{\bra{X\in B_R\;\mbox{or}\;T(X)\in B_R}\:\mbox{and}\; \abs{T\bra{X}-X}>L_\tau} |T \bra{X}-X|^p\\
& \stackrel{\eqref{eq:L2dataestpoint}}{\leq} C L_\tau^p +2^p \tau \bra{1 + C} E_p(4R) + 2^p \bra{C_\tau+\tau C} D_p(4R) \\
& \stackrel{\eqref{eq:L2upperdata},\eqref{eq:L2LleqlnR}}{\leq} 2 \tau \bra{1 + C} E_p \bra{4R} + \bra{2^p(C_\tau + \tau C)+C} \ln^\frac p 2 R.
\end{align*}
Relabeling $\tau$ and $C_\tau$, this implies \eqref{eq:L2basestepdyadic}. Let $R$ be any radius $R \ge r_*$. There exists a dyadic radius $R'\ge R$ satisfying \eqref{eq:L2basestepdyadic} so that\footnote{note that if $R_1<R_2$ are two consecutive approximately dyadic radii, by definition there exists a dyadic radius $R$ such that $R_1 \sim R$ and $R_2\sim 2R$. In particular, $\frac{R_2}{R_1} < 8$} %{\color{red} I do not understand the 32. I guess you want to use sth like $R\leq R' \leq k R$. But why should k be 8 and not 4?}
\[
E_p(R) \le C E_p (R') \le C \tau E_p(4R') + C C_\tau \ln^\frac p2 R' \le C\tau E_p (32 R) + C C_\tau \ln^\frac p2 R,
\]
for some constant $C>0$. Relabeling $\tau$ and $C_\tau$ yields \eqref{eq:L2basestep}. 

\medskip
{\sc Step 3.}
Iteration. Iterating \eqref{eq:L2basestep}, we obtain for any $k\geq 1$
\begin{align*}
 E_p(R) &\leq \tau E_p(32R) + C_\tau \ln^\frac p2 R \\
 & \leq  \tau^2 E_p(32^2R) + \tau C_\tau \ln^\frac p2 R + C_\tau \ln^\frac p2 R\\
 &\leq \tau^k E_p(32^k R) + C_\tau\sum_{l=0}^{k-1}  \tau^l \ln^\frac p2 R \\
 & \stackrel{\eqref{eq:hypharmapprox}}{\leq} \eps \bra{32^p \tau}^k R^p + C_\tau\sum_{l=0}^{k-1}  \tau^l \ln^\frac p2 R.
\end{align*}
%Since by \eqref{eq:hypharmapprox} $E(R)\leq \eps R^2$,  by \eqref{eq:L2upperdata} choosing $\tau$ such that $36 \tau < 1$, we have
%$$ \tau^k E(6^kR) \leq \tau^k \eps 6^{2k} R^2 \to_{k\to\infty} 0.$$
We now fix $\tau$ such that $32^p\tau < 1$ to the effect of
\[
E_p \bra{R} \leq C \sum_{l=0}^{\infty}  \tau^l \ln^\frac p2 R \leq C \ln^\frac p2 R.
\]
%By \eqref{eq:L2upperdata} applied to $6^l R \geq r_*$ instead of $R$ and noting that in particular $\tau < 1$,  we conclude
%$$\sum_{l=0}^{\infty}  \tau^l D(6^l R) \leq C \ln R.$$
\end{proof}

\subsection{Trading Integrability Against Asymptotics}

\begin{lemma}\label{lem:upperboundL1}
Let $p>1$. For every $\eps >0$ there exists a random radius $r_* < \infty$ a.~s. such that
$$ \frac{1}{R^d} \sum_{X\in B_R \;\mbox{or}\;T\bra{X} \in B_R}|T \bra{X} - X|  \leq  \eps  \ln^{\frac12} R.$$
\end{lemma}
\begin{proof}
% By Step 3 we know that for $R\gg 1$
% $$\frac{1}{R^d}\int_{B_R\times\R^d}|x-y|^2 q(dxdy) \leq O(\ln R).$$
% Combining this with the $L^\infty$-estimate Lemma \ref{lem:Linfty} we can estimate for $R\geq r_*$ using the concentration bound Lemma \ref{lem:Poiconcentration} and the Cauchy-Schwarz inequality
By Lemma \ref{lem:harmonicapprox}, we know that there exists a random radius $r_*$ such that for $R\geq r_*$ we have 
\begin{equation}\label{eq:L1L2est}
E_p \bra{R} = \frac{1}{R^d}\sum_{X\in B_R \;\mbox{or}\; T\bra{X} \in B_R}|T\bra{X} - X|^p  \leq C\ln^\frac p2 R.
\end{equation}
Let $0 < \eps \ll 1$. Possibly enlarging $r_*$, we may also assume by Lemma \cite[Lemma 2.1]{HMO21} that there exists a deterministic constant $L$ such that for $R \geq r_*$
\begin{equation}\label{eq:L1largetransport}
\# \bra{\cur{X \in B_R \ | \ \abs{T \bra{X} - X} > L} \cup \cur{T(X) \in B_R \ | \ \abs{T \bra{X} - X} > L}} \leq \eps R^d.
\end{equation}
Furthermore, note that by Lemma \ref{lem:uppdata} and the second and fourth term in the definition of $D_p \bra{R}$ in \eqref{eq:datatermp} we may also assume possibly enlarging $r_*$ again that for $R \geq r_*$ \eqref{eq:L2dataestpoint} holds.
%\begin{equation}\label{eq:L1pointsball}
%\# \bra{\cur{X \in B_R} \cup \cur{ T(X) \in B_R}} \leq C R^d.
%\end{equation}
Finally, we may also assume possibly enlarging $r_*$ that for $R \geq r_*$
\begin{equation}\label{eq:L1Lsmall}
L \leq \eps^\frac1{p'} \ln^\frac12 R.
\end{equation}
We split again the sum into moderate and large transportation distance and apply H\"older's inequality:
\begin{align*}
\frac{1}{R^d}\sum_{X\in B_R \;\mbox{or} \; T\bra{X} \in B_R} |T\bra{X}-X| & \leq  \frac{1}{R^d}\sum_{\bra{X\in B_R\;\mbox{or}\;T(X)\in B_R}\:\mbox{and}\; \abs{T\bra{X}-X}\leq L}  |T\bra{X} - X| \\
& +\frac{1}{R^d}\sum_{\bra{X\in B_R\;\mbox{or}\;T(X)\in B_R}\:\mbox{and}\; \abs{T\bra{X}-X}>L}  |T\bra{X} - X|\\
& \stackrel{\eqref{eq:L2dataestpoint}, \eqref{eq:L1L2est}, \eqref{eq:L1largetransport}}{\leq} C L + \eps^\frac1{p'} E_p(R)^\frac1p \\
&\stackrel{\eqref{eq:L1Lsmall}}{\leq} C \eps^\frac1{p'} \ln^\frac12 R.
\end{align*}
Relabeling $\eps$ proves the claim.
\end{proof}

 \bibliographystyle{abbrv}

\bibliography{OT}
\end{document}